\documentclass[hidelinks,onefignum,onetabnum]{siamart250211}

\usepackage{lipsum}
\usepackage{amsfonts}
\usepackage{graphicx}
\usepackage{epstopdf}
\usepackage{algorithmic}
\ifpdf
  \DeclareGraphicsExtensions{.eps,.pdf,.png,.jpg}
\else
  \DeclareGraphicsExtensions{.eps}
\fi

\newsiamremark{remark}{Remark}
\newsiamremark{hypothesis}{Hypothesis}
\crefname{hypothesis}{Hypothesis}{Hypotheses}
\newsiamthm{claim}{Claim}
\newsiamremark{fact}{Fact}
\crefname{fact}{Fact}{Facts}

\headers{Conductance Estimation in Digraphs}{S. Shao, C. Yang, and X. Ye}

\title{Conductance Estimation in Digraphs:  Submodular Transformation, Lov{\'a}sz Extension  and Dinkelbach Iteration\thanks{Submitted to the editors DATE.
\funding{This work was funded by the National Key R \& D Program of China (No.
2022YFA1005102) and the National Natural Science Foundation of China (Nos. 12325112, 12288101).}}}

\author{Sihong Shao\thanks{CAPT, LMAM and School of Mathematical Sciences, Peking University, Beijing 100871, China
  (\email{sihong@math.pku.edu.cn}).}
\and Chuan Yang\thanks{School of Mathematical Sciences, Peking University, Beijing 100871, China 
  (\email{chuanyang@pku.edu.cn}).}
\and Xinyang Ye\thanks{School of Mathematical Sciences, Peking University, Beijing 100871, China 
  (\email{xinyang@stu.pku.edu.cn}).}
}

\usepackage{amsopn}

\ifpdf
\hypersetup{
  pdftitle={Conductance Estimation in Digraphs:  Submodular Transformation, Lov{\'a}sz Extension  and Dinkelbach Iteration},
  pdfauthor={S. Shao and C. Yang and X. Ye}
}
\fi

\usepackage{subcaption}
\usepackage{booktabs}
\usepackage{amsmath}
\usepackage{amssymb}
\usepackage{cases}

\newtheorem{example}{Example}[section]
\renewcommand{\vec}[1]{\boldsymbol{\mathrm{#1}}}

\providecommand{\vx}{\ensuremath{\vec{x}}}
\providecommand{\vs}{\ensuremath{\vec{s}}}

\providecommand{\vy}{\ensuremath{\vec{y}}}
\providecommand{\vd}{\ensuremath{\vec{d}}}

\providecommand{\vp}{\ensuremath{\vec{p}}}
\providecommand{\vq}{\ensuremath{\vec{q}}}

\providecommand{\va}{\ensuremath{\vec{a}}}
\providecommand{\vb}{\ensuremath{\vec{b}}}
\providecommand{\vu}{\ensuremath{\vec{u}}}
\providecommand{\vv}{\ensuremath{\vec{v}}}
\providecommand{\vy}{\ensuremath{\vec{y}}}

\newcommand{\R}{\mathbb{R}}

\newcommand{\sign}{\mathsf{Sign}}
\newcommand{\sgn}{\mathsf{Sgn}}
\newcommand{\vol}{\mathrm{Vol}}

\newcommand{\cut}{\mathrm{cut}}

\newcommand{\argmin}{\mathop{\mathrm{argmin}}}

\DeclareMathOperator*{\argmax}{\ensuremath{\mathrm{argmax\,}}}
\DeclareMathOperator*{\nen}{\ensuremath{\mathrm{NEN\,}}}

\begin{document}
\maketitle
\begin{abstract}
Conventional spectral digraph partitioning methods typically symmetrize the adjacency matrix, thereby transforming the directed graph partitioning problem into an undirected one, where bipartitioning is commonly linked to minimizing graph conductance. However, such symmetrization approaches disregard the directional dependencies of edges in digraphs, failing to capture the inherent imbalance crucial to directed network modeling. 
Building on the parallels between digraph conductance and conductance under submodular transformations, we develop a generalized framework to derive their continuous formulations. By leveraging properties of the Lov{\'a}sz extension, this framework addresses the fundamental asymmetry problem in digraph partitioning. We then formulate an equivalent fractional programming problem, relax it via a three-step Dinkelbach iteration procedure, and design the Directed Simple Iterative ($\mathbf{DSI}$) algorithm for estimating digraph conductance. The subproblem within $\mathbf{DSI}$ is analytically solvable, and the algorithm is guaranteed to converge provably to a binary local optimum. Extensive experiments on synthetic and real-world networks demonstrate that our $\mathbf{DSI}$ algorithm significantly outperforms several state-of-the-art methods in digraph conductance minimization.
\end{abstract}



\section{Introduction}
\label{sec:intro}
Graph conductance is a crucial criterion to evaluate the quality of graph partitioning \cite{2006localPagerank,Benson2016HigherorderOO}, with applications in community detection and many other fields \cite{2000NomalizedCutImage,2013SpectralSparse}.
However, conductance is traditionally defined only for simple undirected networks, and algorithms for computing conductance in digraphs (directed graphs) have received less attention. Yet, many real-world networks are naturally represented as digraphs. In biological networks, digraphs signify carbon flow directions or connections between neurons \cite{Kaiser2006NonoptimalCP,Benson2016HigherorderOO}. In social media networks, they indicate the influence of podcast news on political parties or relationships between individuals \cite{2005politicalblog,2007emailcore,2017localhighorder}. In economic sciences, arcs represent transportation reachability between cities \cite{2007reachability}.

For a weighted digraph $G=(V,A)$ and a partition $(S,\Bar{S})\, (\Bar{S} =V\backslash S, \,\,S \text{ and } \Bar{S}$ are both nonempty), where $V$ is the vertex set, $A$ is the arc set, the conductance $\varphi_D(S)$ of a subset $S$ in a digraph is  defined as
\begin{equation}\label{eq:varphi_D}
\varphi_D(S)=\frac{\min\{\cut^+(S),\cut^-(S)\}}{\min\{\vol(S),\vol(\Bar{S})\}},
\end{equation} 
which is used to evaluate the quality of digraph partitioning \cite{2016nonlinearlap}.
Here, $\cut^+(S)$ and $\cut^-(S)$ are the sum of weights of arcs leaving and entering $S$, respectively and $\cut(S) = \min\{\cut^+(S), \cut^-(S)\} .$ 
For an arc $e = i \to j,$ where $i$ is the arc starts from and $j$ is the arc ends to, denote the weight of this arc as $w_{ij}$ or $w_e.$
Set
 \begin{align*}
 \partial^+(S) = \{ i\to j \in A \mid i \in S, j \in V \setminus S\}, \text{ and }
 \partial^-(S) = \{ i\to j \in A \mid i \in V \setminus S, j \in S\}.
 \end{align*}
Then, $\cut^+(S) = \sum\limits_{e\in \partial^+(S)} w_e,$ $\cut^-(S) = \sum\limits_{e \in \partial^-(S)} w_e.$
The volume of $S$, denoted by $\vol(S)$, is defined as $\sum\limits_{v \in S} d_v,$ where $d_v$ is the sum of weights of all the arcs incident to vertex $v$ (containing both arcs leaving and entering $v$).
The out-degree $d^+_v$ and the in-degree $d^-_v$ of a vertex $v$ are defined as the sum of weights of arcs leaving and entering $v$, respectively. Then, the degree of $v$ is $d_v = d^+_v + d^-_v$ and the delta degree of $v$ is defined as $d_v^{\delta}=d^+_v-d^-_v $. To simplify notation, denote the corresponding degree vectors as $\vec d,$ $\vec d^+,$ $\vec d^-$ and $\vec d^{\delta}.$
In the same way, for a nonempty subset $ S\subsetneq V,$ there are out-conductance $\varphi^+(S)= \frac{\cut^+(S)}{\min\{\vol(S),\vol(\Bar{S})\}}$ and in-conductance $ \varphi^-(S)=\frac{\cut^-(S)}{\min\{\vol(S),\vol(\Bar{S})\}}$.
The conductance, in-conductance and out-conductance of a digraph $G$ are $\varphi_D(G) = \min\limits_{\emptyset \neq S \subsetneq V} \varphi_D(S),$ $\varphi^-(G) = \min\limits_{\emptyset \neq S \subsetneq V} \varphi^-(S),$ and $\varphi^+(G) = \min\limits_{\emptyset \neq S \subsetneq V}\varphi^+(S),$ respectively. 
In fact, it's easy to show that 
\begin{equation}\label{eq:min formula}
\varphi_D(G)=\min\{ \varphi^+(G),\varphi^-(G)\}.
\end{equation} We prove \eqref{eq:min formula} in Appendix \ref{proof of min formula}.

Many approaches have been proposed for spectral partitioning of digraphs. For example, the stationary distribution $\pi$ of the random walk process on digraphs can be used to define the Laplacian matrix \cite{Chung2005LaplaciansAT}, where the corresponding symmetrized adjacency matrix is $M = \frac{\Pi P + P^{T} \Pi}{2}$. Here, $P$ denotes the transition probability matrix of the random walk and $\Pi = \text{Diag}(\pi).$
Alternatively, there are several other symmetrization methods \cite{2011digraphSymmetic}, including $M = \frac{P + P^{T}}{2},$  $M=PP^T$ and $M = B + C,$ where 
$B = D_o^{-\alpha} P D_i^{-\beta} P^T D_o^{-\alpha}$, 
$C = D_i^{-\beta} P^T D_o^{-\alpha} P D_i^{-\beta},$ $D_o = \text{Diag}( \vec d^+),$ $D_i = \text{Diag}( \vec d^-),$ $\alpha,\beta$ are some parameters to be chosen.
However, these methods try to symmetrize the adjacency matrix of a digraph, since the eigenvector computation depends highly on the symmetry property of the adjacency matrix. 
These methods disregard the orientational dependencies of edges in digraphs, thus failing to model the intrinsic asymmetry of directed networks.
Hermite matrices were introduced in \cite{2020hermite_cluster} to represent adjacency relations in digraphs, thereby providing clustering algorithms for digraphs. However, this approach lacks theoretical guarantees.
A nonlinear Laplacian for digraphs and a heat kernel method for computing eigenvectors of the digraph adjacency matrix were developed in \cite{2016nonlinearlap}.
Through relaxation, it converts the conductance computing problem to the minimization of Rayleigh quotient $\mathcal{R}_G(\vx) = \frac{\vx^T \mathcal{L}_G(\vx) }{\vx^T \vx}$ for all $\vx\in \mathbb{R}^n \backslash \{0\}$ that are orthogonal to $\sqrt{\vec d/\vol(V)},$ where $\mathcal{L}_G$ is a nonlinear Laplacian operator.
However, this method needs to round the eigenvector to get the final partition, which is hard to obtain a local optimum.

Instead of using relaxation, we propose an equivalent formula of digraph conductance via the  Lov{\'a}sz extension.
We claim that
\begin{equation}\label{eq: main equivalent}
\min\limits_{\vx \in \mathbb{R}^n \backslash \{ t \vec 1\}, t \in R} r(\vx) =\min\limits_{\emptyset \neq S \subsetneq V } \varphi_D(S),
\end{equation}
where
\begin{equation} 
\label{eq:dir ratio aim fun}
r(\vx)=\frac{\vol(V)\|\vx\|_\infty -I^+(\vx)-J(\vx)}{ 2N(\vx)},
\end{equation}
and
\begin{align}
I^+(\vx)&=\sum\limits_{i \to j \in A} {w_{ij}}|x_i+x_j|, \label{eq: Iplus}\\
J(\vx)&=|\sum\limits_{i \to j \in A}w_{ij}(x_i-x_j)|=|\sum\limits_{i \in V} d_{i}^\delta x_i |,\label{eq:Jx}\\
N(\vx)&= \min\limits_c \sum_{i\in V} d_i |x_i -c|.\label{eq:Nd}
\end{align}
Detailed proof of \eqref{eq: main equivalent} will be given in Section \ref{sec framework}. 
Motivated by the similarities between digraph conductance and conductance under submodular transformations, we propose a generalized framework to derive the equivalent continuous formulations.
Details will be shown in Section \ref{sec framework} (Theorem \ref{them:min frac lovasz}).

In this paper, we equivalently transform the fractional formula \eqref{eq: main equivalent} to a two-step iterative Dinkelbach  framework \eqref{eq:iterative alg two-step}, whose global optimal solution is equivalent to the minimum of \eqref{eq: main equivalent}.
However, the subproblem \eqref{eq:sub two-step} is NP-hard.
Through subgradient relaxation of \eqref{eq:sub two-step}, \eqref{eq:iterative alg two-step} can be modified to a three-step one \eqref{ite:three-step}.
Here, the subproblem \eqref{eq:subproblem} is analytically solvable and the sequence $\{r^k\}$ \eqref{eq:r sequence} is strictly decreasing with the special subgradient selection step (\eqref{eq；subg} and Algorithm \ref{alg:subgradient selection}). 
Based on the above properties, we show that our $\mathbf{DSI}$ algorithm will converge to a binary local optimum in finite steps.
We apply the $\mathbf{DSI}$ algorithm to various digraphs and demonstrate an improvement in conductance computation over some other methods.
In summary, our paper develops a simple and flexible method for computing digraph conductance with local convergence guarantee.
The $\mathbf{DSI}$ algorithm can also be generalized to estimate other conductance under submodular transformations (see Definition \ref{def:submodular transformations}).
By moving beyond the symmetric Laplacian of digraphs, our work opens new avenues for directed network analysis.

This paper is structured as follows.
In Section \ref{sec framework}, we will prove \eqref{eq: main equivalent} and illustrate it in a more generalized framework utilizing submodular transformations.
In Section \ref{sec DSI}, we will present the detailed steps of our $\mathbf{DSI}$ algorithm for computing conductance in digraphs.
We will demonstrate that our algorithm exhibits a favorable convergence property (Theorem \ref{them:local convergence}), specifically converging to a local optimum within a finite number of steps.
Some proofs will be given in the Appendix for clarity.
In Section \ref{sec exp}, we will compare our $\mathbf{DSI}$ algorithm with other algorithms, including Sweepcut, on both synthetic dataset and real networks.

\section{A General Framework for Seeking Equivalent Continuous Formulas}\label{sec framework}

The submodular transformation was proposed to unify and generalize the Cheeger inequalities for undirected graphs, directed graphs, and hypergraphs \cite{2019cheeger_submodular}. 
Under this framework, we derive continuous formulations of \eqref{eq:phi_F}, which inherently encompass the digraph conductance defined in \eqref{eq:varphi_D}.
We first give the definition of submodular transformation in Section \ref{subsec: submodular}.
Then, we introduce an important tool, Lov{\'a}sz extension in Section \ref{subsec: lovasz}.
Finally, we show how to find equivalent objective functions on multiple graph models in Section \ref{subsec: application}.


\subsection{Conductance under Submodular Transformations}\label{subsec: submodular}

Let $V$ be a set of size $n$ and $P(V) = \{S : S \subseteq V\}.$ $ f: P(V) \rightarrow \mathbb{R},$ is called \textbf{submodular} if for every $ S, T \subseteq V $, we have $ f(S) + f(T) \geq f(S \cap T) + f(S \cup T) $.
Note that the \textbf{cut function} $ \cut_G : \{0, 1\}^V \to \mathbb{R} $ associated with an undirected graph (a digraph) $ G $ is submodular, where $\cut_G(S) $ for a vertex set $ S \subseteq V $ represents the sum of weights of edges (arcs) leaving $ S $ and entering $ V \setminus S $. 
The submodular transformation was proposed by Yoshida \cite{2019cheeger_submodular} to give a general framework for graph conductance.
\begin{definition}[Submodular transformation \cite{2019cheeger_submodular}]\label{def:submodular transformations}
Let $G = (V, E)$ be a graph, $V$ and $E$ are its vertex set and edge set, respectively. Let $P(V) = \{S : S \subseteq V\},$ and $ F: P(V) \rightarrow \mathbb{R}^E$ is called \textbf{submodular transformation} if for each $ e \in E $, the induced component function $F_e: P(V) \to \mathbb{R}$ is submodular, where for any $S\subset V, $ $ F_e(S)$ is the corresponding component of $F(S)$ related to $e$.
\end{definition}

\begin{example}[Digraph]
Let $G=(V,A)$ be a digraph. A submodular transformation $F:\{0,1\}^V\rightarrow\mathbb{R}^A$ relating to cut value can be obtained as follows. For each arc $e\in A$, let the corresponding component $F_e:P(V)\rightarrow\mathbb{R}$ denote the cut function of the digraph with the single arc $e$, where $F_e(S) =1$ if $e\in \cut^+(S)$ and  $F_e(S) =0$  otherwise. 
\end{example}
\begin{remark}
Note that the cut function $\cut_G(S)=\min\{ \cut^+(S),\cut^-(S)\}$ is not submodular. 
\end{remark}


For a partition $(S,\Bar{S})$ of $V$ and a submodular transformation $F,$ the conductance of $S$ is defined as:
\begin{equation}
\label{eq:phi_F}
 \varphi_F(S) = \frac{\min\{\cut_F(S), \cut_F(V \backslash S)\}}{\min\{\vol(S), \vol(V \backslash S)\}},
 \end{equation}
where $\cut_F(S) = \sum\limits_{e\in E} F_e(S),$
and $\varphi_F(G) = \min\limits_{\emptyset \neq S \subsetneq V} \varphi_F(S).$

\subsection{Lov{\'a}sz Extension and Beyond}\label{subsec: lovasz}

The main technique we use to find a continuous equivalent objective function of $\varphi_F(G)$ 
is Lov{\'a}sz extension \cite{Lovsz1982_SubmodularFA,Bac2013_learn_submodular_fun}, which
 is an important tool in establishing a bridge between discrete optimization and continuous optimization, thereby allowing the application of continuous optimization methods to solve discrete optimization problems \cite{2013_constrained_fractional,2018Lovszsetpair,2024maxcutSI}.

\begin{definition}[Lov{\'a}sz extension]\label{def:lovasz extension}
For $\vx \in \mathbb{R}^n,$ let $ \sigma: [n]\cup \{0\} \to [n] \cup \{0\} $ be a bijection such that $x_{\sigma(1)} \leq x_{\sigma(2)} \leq \cdots \leq x_{\sigma(n)} $ and $ \sigma(0) = 0,$ where $ x_0 := 0$ and $[n]=\{1,2,...,n\}.$ 
Define the sets $V_{\sigma(i)} := \{j \in V : x_j > x_{\sigma(i)}\},$ for $ i = 1, \ldots, n - 1,$ and $ V_0 = V.$
Let $V$ be a set of size $n$ and $P(V) = \{S : S \subseteq V\}.$ Given $ f: P(V) \rightarrow [0, +\infty),$ the Lov{\'a}sz extension $F^L$ of $f $ is a mapping from $ \mathbb{R}^n $ to $ \mathbb{R}$ defined by
$$F^L(\vx) = \sum_{i=0}^{n-1} (x_{\sigma(i+1)} - x_{\sigma(i)}) f(V_{\sigma(i)}).$$
\end{definition}

\begin{proposition}\label{prop:lovasz extension}[\cite{Bac2013_learn_submodular_fun}, Definition 3.1]
Let $V$ be a set of size $n$ and $P(V) = \{S : S \subseteq V\}.$ Given $ f: P(V) \rightarrow [0, +\infty),$ 
then the Lov{\'a}sz extension $F^L$ of $f$ can be calculated as the following:
\begin{equation}
F^L(\vx) = \int^{\max\limits_{1 \leq i \leq n} x_i } _{\min\limits_{1 \leq i \leq n} x_i } f(V_t(\vx)) \, dt + f(V)\min\limits_{1 \leq i \leq n} x_i,
\label{eq: Lov{a}sz extension prop}
\end{equation}
where $V_t(\vx)=\{i\in V: x_i>t\}$.
\end{proposition}

Let $\mathbf{1}_S\in \mathbb{R}^n$ be the indicator vector of the set $S$, that is, the $i$-th component of $\mathbf{1}_S$ is 1 if and only if $v_i \in S$. Hence, for any $S\subseteq V$, we have $f(S) = F^L(\mathbf{1}_S)$. 
We first show some basic properties of Lov{\'a}sz extension.

\begin{lemma}
\label{lem:min Lovasz}
Let $V$ be a set of size $n$ and $P(V) = \{S : S \subseteq V\}.$ Given $g,$ $h: P(V) \rightarrow [0, +\infty),$ and define $f(S) = \min\{g(S), h(S)\},$ $\forall \, S\subset V.$ $F^l(\vx),$ $G^L(\vx),$ $H^L(\vx)$ are Lov{\'a}sz extensions of $f,$ $g,$ $h,$ respectively. Then $F^L$ is upper bounded by the minimum of $G^L$ and $H^L,$ that is,
$$F^L(\vx) \leq \min\{G^L(\vx), H^L(\vx)\}.$$
\end{lemma}
\begin{proof}

From Definition \ref{def:lovasz extension}, 

\begin{align*}
F^L(\vx) &= \sum_{i=0}^{n-1} (x_{\sigma(i+1)} - x_{\sigma(i)}) f(V_{\sigma(i)})\\
&= \sum_{i=0}^{n-1} (x_{\sigma(i+1)} - x_{\sigma(i)}) \min\{g(V_{\sigma(i)}),h(V_{\sigma(i)})\}\\
&\leq \min\left\{ \sum_{i=0}^{n-1} (x_{\sigma(i+1)} - x_{\sigma(i)}) g(V_{\sigma(i)}),\sum_{i=0}^{n-1} (x_{\sigma(i+1)} - x_{\sigma(i)}) h(V_{\sigma(i)})\right\}\\
&=\min\{G^L(\vx), H^L(\vx)\}.
\end{align*}
\end{proof}

\begin{theorem}\label{them:min frac lovasz}
Let $V$ be a set of size $n$ and $P(V) = \{S : S \subseteq V\}.$ Given $f_1,$ $f_2,$ $g: P(V) \rightarrow [0, +\infty),$ $g(\emptyset)=g(V)=0,$ and $\forall \, \emptyset \neq S \subsetneq V,$ $g(S)>0.$ Define $f(S) = \min\{f_1(S), f_2(S)\},$ $\forall \, S\subset V.$
$F^L,$ $F^L_1,$ $F^L_2,$ $G^L$ are Lov{\'a}sz extensions of $f,$ $f_1,$ $f_2,$ $g,$ respectively.
Set $\Tilde{F}^L = \min\{F_1^L, F_2^L\},$ then for any set $ S\subset V $, $ F^L(\mathbf{1}_S) = \Tilde{F}^L(\mathbf{1}_S) $.
If $f_1(V)=f_2(V)=0,$ then it holds that
$$\min_{\emptyset \neq S \subsetneq V} \frac{f(S)}{g(S)} = \inf_{\vx \in \mathbb{R}^n \backslash \{ t \vec 1\}, t \in R} \frac{F^L(\vx)}{G^L(\vx)}
 = \inf_{\vx \in \mathbb{R}^n \backslash \{ t \vec 1\}, t \in R} \frac{\Tilde{F}^L(\vx)}{G^L(\vx)}.$$
\end{theorem}

\begin{proof}

For any set $S\subset V$, 
$$F^L(\mathbf{1}_S) = f(S)=\min\{f_1(S),f_2(S)\} =\min\{ F_1^L(\mathbf{1}_S),F_2^L(\mathbf{1}_S)\}=\Tilde{F}^L(\mathbf{1}_S),$$
which proves $F^L(\mathbf{1}_S) =\Tilde{F}^L(\mathbf{1}_S).$
For any $\vx \in \mathbb{R}^n \backslash \{t \vec 1\}, t \in R,$ the corresponding set $V_{\sigma(i)}$ $(1\leq i \leq n-1)$ is not empty and is a proper subset of $V,$ which leads to $g(V_{\sigma(i)})>0,$ $\forall \, 1\leq i \leq n-1$ and $G^L(\vx)>0.$
Given that $f(V)=\min\{f_1(V),f_2(V)\}= g(V)=0,$  for any $\vx \in \mathbb{R}^n \backslash \{t \vec 1\}, t \in R,$ we have
\begin{align*}
F^L(\vx)&=\sum_{i=1}^{n-1} f(V_{\sigma(i)})\left(x_{\sigma(i+1)}-x_{\sigma(i)}\right).\\
&=\sum_{i=1}^{n-1}\frac{f(V_{\sigma(i)})}{g(V_{\sigma(i)})}g(V_{\sigma(i)})\left(x_{\sigma(i+1)}-x_{\sigma(i)}\right) \\
&\geq\min_{j=1,...,n-1}\frac{f(V_{\sigma(j)})}{g(V_{\sigma(j)})}\left(\sum_{i=1}^{n-1}g(V_{\sigma(i)}) (x_{\sigma(i+1)}-x_{\sigma(i)})\right)\\
&=\min_{j=1,...,n-1}\frac{f(V_{\sigma(j)})}{g(V_{\sigma(j)})} G^L(\vx).
\end{align*}
By Lemma \ref{lem:min Lovasz},
\begin{equation*}
\frac{\Tilde{F}^L(\vx)}{G^L(\vx)}\geq 
\frac{F^L(\vx)}{G^L(\vx)}\geq \min\limits_{j=1,...,n-1} \frac{f(V_{\sigma(j)})}{g(V_{\sigma(j)})}.
\end{equation*}
On the one hand,
$$\inf_{\vx \in \mathbb{R}^n \backslash \{ t \vec 1\}, t \in R}\ \frac{\Tilde{F}^L(\vx)}{G^L(\vx)}\geq
\inf_{\vx \in \mathbb{R}^n \backslash \{ t \vec 1\}, t \in R}\frac{F^L(\vx)}{G^L(\vx)}
\geq\min\limits_{j=1,...,n-1} \frac{f(V_{\sigma(j)})}{g(V_{\sigma(j)})}
\geq\min_{\emptyset \neq S \subsetneq V}\frac{f(S)}{g(S)}. $$
On the other hand, we have $\Tilde{F}^L(\mathbf{1}_S)=F^L(\mathbf{1}_S)$ for any set $S\subset V$, thus leading to
$$\min_{\emptyset \neq S \subsetneq V}\frac{f(S)}{ g(S)}
=\min_{\emptyset \neq S \subsetneq V}\frac{F^L(\mathbf{1}_S)}{G^L(\mathbf{1}_S)}
=\min_{\emptyset \neq S \subsetneq V}\frac{\Tilde{F}^L(\mathbf{1}_S)}{G^L(\mathbf{1}_S)}
\geq\inf_{\vx \in \mathbb{R}^n \backslash \{ t \vec 1\}, t \in R}\frac{\Tilde{F}^L(\vx)}{G^L(\vx)}.$$
Therefore,
$$\min_{\emptyset \neq S \subsetneq V}\frac{f(S)}{ g(S)}
=\inf_{\vx \in \mathbb{R}^n \backslash \{ t \vec 1\}, t \in R}\frac{\Tilde{F}^L(\vx)}{G^L(\vx)}. $$
\end{proof}

\begin{remark}\label{remark: moreH}
In fact, for any function $ H : \mathbb{R}^n \rightarrow \mathbb{R}$ satisfying $ H(\mathbf{1}_S) = F^L(\mathbf{1}_S)$ for any $ S\subsetneq V$ and $H(\vx) \geq F^L(\vx)$ for any $\vx\in \mathbb{R}^n$, 
$$\inf_{\vx \in \mathbb{R}^n \backslash \{ t \vec 1\}, t \in R}\frac{H(\vx)}{G^L(\vx)}=\inf_{\vx \in \mathbb{R}^n \backslash \{ t \vec 1\}, t \in R}\frac{F(\vx)}{G^L(\vx)}=\min_{\emptyset \neq S \subsetneq V}\frac{f(S)}{ g(S)}$$
holds, where $f,$ $g,$ $F^L,$ $G^L$ are defined in Theorem \ref{them:min frac lovasz}.
\end{remark}


Theorem \ref{them:min frac lovasz} provides steps to identify equivalent continuous objective functions of \eqref{eq:phi_F}, which are not unique by Remark \ref{remark: moreH}.  Define
\begin{equation}\label{eq: set F}
\mathcal{F}(\varphi_F)= \left\{ f:\mathbb{R}^n \to \R| \min\limits_{\vx \in \mathbb{R}^n \backslash \{ t \vec 1\}, t \in R} f(\vx) =\varphi_F(G) \right\}
\end{equation}
as a set of all the equivalent continuous optimizations for $\varphi_F(G)$.

\begin{corollary} \label{eq；conclusion}
From Remark \ref{remark: moreH}, for any function $F(\vx)$ satisfying
\begin{equation}\label{eq:equal property}
F(\vx)\geq \min\{ \cut_F^{+L}(\vx),\cut_F^{-L}(\vx)\} \text{ and }  F(\mathbf{1}_S)=\min\{\cut^+_F(S),\cut^-_F(S)\}
\end{equation}
for any $\vx\in\mathbb{R}^n$ and $S\subset V$, the following equality holds:
 $$\varphi_F(G) = \min\limits_{\vx \in \mathbb{R}^n \backslash \{ t \vec 1\}, t \in R}\frac{F(\vx)}{N(\vx)},$$ 
where $N(\vx)$, $\cut^{+L}(\vx)$, $\cut^{-L}(\vx)$ are the Lov{\'a}sz extensions of $\min\{\vol(S), \vol(\Bar{S})\},$ $\cut^+(S)$, $\cut^-(S),$  respectively.
Therefore, we have $\frac{F(\vx)}{N(\vx)}\in\mathcal{F}(\varphi_F)$.
\end{corollary}


\subsection{Applications} \label{subsec: application}
We now provide more than one equivalent objective function of digraph conductance and undirected graph conductance using Corollary \ref{eq；conclusion}.

\paragraph{Directed Conductance}\label{para:digraph}
Now, we will use Theorem \ref{them:min frac lovasz} to explore equivalent continuous objective functions for $\varphi_D(S)$ \eqref{eq:varphi_D} and prove \eqref{eq: main equivalent}. 
First, we calculate the Lov{\'a}sz extension $\cut^{+L}(\vx)$ for $\cut^+(S)$.

\begin{lemma}
Let $G=(V,A)$ be a digraph, 
the Lov{\'a}sz extension of $\cut^+(S)$ is
$$\cut^{+L}(\vx) = \sum_{i \to j \in A} \frac{w_{ij}}{2}(x_i - x_j + |x_i - x_j|).$$
\end{lemma}

\begin{proof}
Given $\vx\in\mathbb{R}^n,$ and set $V_t(\vx) = \{i \in V : x_i > t\}$, the sum of the weights of the arcs leaving $V_t(\vx)$ is
$$\cut^+(V_t(\vx)) = \sum_{i \to j \in A} w_{ij} \chi_{x_j \leq t < x_i}.$$
Since $\cut^+(V) = 0,$ and by Proposition \ref{prop:lovasz extension}, we obtain
\begin{align*}
\cut^{+L}(\vx) &= \sum_{i \to j \in A} w_{ij} \int^{\max\limits_{1 \leq i \leq n} x_i } _{\min\limits_{1 \leq i \leq n} x_i } \chi_{x_j \leq t < x_i} dt \\
&= \sum_{i \to j \in A} w_{ij} \max\{0, x_i - x_j\} \\
&= \sum_{i \to j \in A} \frac{w_{ij}}{2}(x_i - x_j + |x_i - x_j|)
.
\end{align*}
\end{proof}
Similarly, for $\cut^-({S})$, its Lov{\'a}sz extension is
$$\cut^{-L}(\vx) = \sum_{i \to j \in A} \frac{w_{ij}}{2}(x_j - x_i + |x_j - x_i|).$$
For $\cut(S) = \min\{\cut^+(S), \cut^-({S})\}$, its Lov{\'a}sz extension $\cut^L(\vx)$ is not linear for each arc and is not additive. Instead of directly calculating $\cut^L(\vx)$, we notice it can be upper bounded by $\cut^{+L}(\vx)$ and $\cut^{-L}(\vx).$ Under condition \eqref{eq:equal property}, the upper bound function will not violate the equivalent condition.
From Lemma \ref{lem:min Lovasz}, for all $\vx\in\mathbb{R}^n$, we have

\begin{align}
\cut^L(\vx)&\leq\min\{\cut^{+L}(\vx),\cut^{-L}(\vx)\}=\frac{1}{2}(I(\vx)-J(\vx)):=G(\vx),\\
G(\vx)&\leq \frac{1}{2}\left(\vol(V)\|\vx\|_\infty -I^+(\vx) -J(\vx)\right):=F(\vx),\label{eq: F(x)}
\end{align}
where $I(\vx) =\sum\limits_{i \to j \in A} {w_{ij}}|x_i-x_j|.$ 
It's easy to check that for any set $ S \subset V $, $ F(\mathbf{1}_S) = \min\{\cut^+(S),\cut^-(S) \}$,
thus according to Corollary \ref{eq；conclusion}, we have
\begin{equation*}
\min\limits_{\vx \in \mathbb{R}^n \backslash \{ t \vec 1\}, t \in R} \Tilde{r}(\vx) = \min\limits_{\vx \in \mathbb{R}^n \backslash \{ t \vec 1\}, t \in R} r(\vx) = \min\limits_{\emptyset \neq S \subsetneq V} \varphi_D(S),
\end{equation*}
where $r(\vx)=\frac{F(\vx)}{N(\vx)}$, $\Tilde{r}(\vx)=\frac{G(\vx)}{N(\vx)}$, and now
finishes the proof of \eqref{eq: main equivalent}.

\paragraph{Single Directed Conductance}
Given a digraph $G=(V,A),$ the equivalent continuous formulations for the out-conductance $\varphi^+(G)$ mentioned in Section~\ref{sec:intro} are
$$\min\limits_{\vx \in \mathbb{R}^n \backslash \{ t \vec 1\}, t \in R}\frac{\vol(V)\|\vx\|_\infty -I^+(\vx)-J_0(\vx)}{ 2N(\vx)}=\min\limits_{\vx \in \mathbb{R}^n \backslash \{ t \vec 1\}, t \in R}\frac{I(\vx)-J_0(\vx)}{ 2N(\vx)}=\varphi^+(G),$$
where $ J_0(\vx)=\sum\limits_{i \to j \in A}w_{ij}(x_i-x_j)=\sum_{i} d_{i}^\delta x_i$. For $\varphi^-(G),$ it is nearly the same.

\paragraph{Undirected Conductance \cite{Shao2024SI_cheeger,Chang20151Lapalg}}
Given an undirected graph $G=(V,E)$, let $\{i,j\}\in E$ denote the edge connecting $i\in V$ and $j\in V$. For any nonempty set $S\subsetneq V$, 
let $\cut(S)$ denote the sum of the weights of all edges crossing the partition $(S,\bar{S})$. The conductance of $G=(V,E)$ is defined as
\begin{equation}
 \varphi(G)=\min\limits_{\emptyset \neq S \subsetneq V}\varphi(S),\quad \varphi(S) = \frac{\cut(S)}{\min\{\vol(S), \vol(V \backslash S)\}}. 
 \end{equation}
By Corollary \ref{eq；conclusion}, the two equivalent continuous formulations of $\varphi(G)$ are
\begin{equation}
\label{eq:undirected}
    \min\limits_{\vx \in \mathbb{R}^n \backslash \{ t \vec 1\}, t \in R}\frac{\vol(V)\|\vx\|_\infty -\sum\limits_{\{i,j\} \in E} {w_{ij}}|x_i+x_j|}{ N(\vx)}=\min\limits_{\vx \in \mathbb{R}^n \backslash \{ t \vec 1\}, t \in R}\frac{\sum\limits_{\{i,j\} \in E} {w_{ij}}|x_i-x_j|}{ N(\vx)}.
\end{equation}
Given the non-uniqueness of equivalent continuous objective functions in $\mathcal{F}(\varphi_F)$, the selection of $r \in \mathcal{F}(\varphi_F)$ critically impacts computational efficiency. As demonstrated in the numerical experiments of \cite{Shao2024SI_cheeger}, iterative algorithms based on the first equivalent formulation of \eqref{eq:undirected} achieve significantly faster convergence than those using the second one. Motivated by this efficiency hierarchy, we will design iterative algorithm by adopting
$\min\limits_{\vx \in \mathbb{R}^n \backslash \{ t \vec 1\}, t \in R} r(\vx)$ \eqref{eq: main equivalent} for solving digraph conductance $\varphi_D(G)$.

\section{A Simple Iterative Algorithm}\label{sec DSI}

In this section, we elaborate on the $\mathbf{DSI}$ algorithm, which employs a continuous optimization approach to solve the discrete optimization problem.
Notably, this strategy has been applied to various graph cut problems,  including Max Cut \cite{2024maxcutSI}, AntiCheeger Cut \cite{2021AntiCheeger}, and Cheeger Cut \cite{Shao2024SI_cheeger} in undirected graphs.
From \eqref{eq: main equivalent}, we know that 
$$\min\limits_{\vx \in \mathbb{R}^n \backslash \{ t \vec 1\}, t \in R}{r}(\vx) = \min\limits_{\emptyset \neq S \subsetneq V} \varphi_D(S),$$
where $r(\vx)$ is defined in \eqref{eq:dir ratio aim fun}.
Dinkelbach iterative algorithm \cite{1967Dinkelbach} was designed to solve fractional programming and here we use it to solve $\min\limits_{\vx \in \mathbb{R}^n \backslash \{ t \vec 1\}, t \in R} r(\vx).$
For notational convenience, we define
$\Omega_p^1 =\{ \vx \in \mathbb{R}^n: \max\limits_i x_i +\min\limits_i x_i = 0,\, \|\vx\|_p=1 \}$ and $\Omega_p^2 =\{ \vx \in \mathbb{R}^n: \min \limits_i |x_i| = 0,\, \|\vx\|_p=1 \}.$
Set
\begin{equation}\label{eq:Q}
Q_{r}(\vx)= \frac{I^+(\vx)+ J(\vx)+2r N(\vx) }{ \vol(V) }.
\end{equation}
Note that $\Omega_p^1 \cup \Omega_p^2$ is a compact set, and 
from Dinkelbach's algorithm, 
the sequence $\{\vx^k\}$ generated by the iterative steps
\begin{subequations}
\label{eq:iterative alg two-step}
\begin{numcases}{}
\vx^{k+1}=\argmin\limits_{\vx\in\Omega_p^1 \cup \Omega_p^2} 
\{\|\vx\|_\infty -Q_{r^k}(\vx)\},~p\in[1,\boldsymbol{\infty}], \label{eq:sub two-step}\\
r^{k+1}= r(\vx^{k+1}).
\end{numcases}
\end{subequations}
will converge to the optimum of $r(\vx),$ whose value is denoted as $r_{\min}.$ 

\begin{theorem}\label{them:convergence}
The sequence $\{r^{k}\}$ generated by \eqref{eq:iterative alg two-step} will converge to the global optimum $r_{\min}.$ 
\end{theorem}
\begin{proof}

It indicates $r_{\min}=\varphi_D(G)$. The definition of $\vx^{k + 1}$ in \eqref{eq:sub two-step} implies
$$0=\|\vx^{k}\|_{\infty}- Q_{r^{k}}(\vx^{k})
\geq\|\vx^{k + 1}\|_{\infty}-Q_{r^{k}}(\vx^{k + 1}),$$
which means $$ r^k \geq \frac{\vol(V)\|\vx^{k+1}\|_\infty - I(\vx^{k+1})- J(\vx^{k+1})}{2 N(\vx^{k+1}) }= r^{k + 1}\geq r_{\min},\ \forall\, k\in\mathbb{N}^{+}.$$
Thus we have
$$\exists\, r^{*}\in[r_{\min},r^{1}]\text{ s.t. }\lim_{k\rightarrow+\infty}r^{k}=r^{*}.$$
It suffices to show that $r_{\min} \geq r^{*}$. 
We denote
$$f(r)=\min_{\vx \in \Omega_p^1 \cup \Omega_p^2}(\|\vx\|_{\infty}-Q_r(\vx)),$$
which must be continuous on $\mathbb{R}$ by the compactness of $\Omega_p^1\cup \Omega_p^2$. 
Note that
\begin{align*}
f(r^k)&=\|\vx^{k + 1}\|_{\infty} - Q_{r^k}(\vx^{k + 1})\\
&= \| \vx^{k + 1}\|_{\infty} -\frac{I^+(\vx^{k+1})+ J(\vx^{k+1})+2r^k N(\vx^{k+1}) }{ \vol(V) }\\
&= \frac{2N(\vx^{k+1})}{\vol(V)} \left(\frac{\vol(V) \| \vx^{k + 1}\|_{\infty} -I^+(\vx^{k+1})- J(\vx^{k+1}) }{2N(\vx^{k+1})}-r^k \right)\\
&=\frac{2N(\vx^{k + 1})}{\vol(V)}\cdot(r^{k + 1}-r^k)\to 0\quad\text{as}\quad k\to+\infty,
\end{align*}
we have
$$f(r^*)=\lim_{k\to+\infty}f(r^k)=0,$$
and thus
$$\|\vx\|_{\infty} - Q_{r^*}(\vx)\geq 0,\, \forall\,\vx \in \Omega_p^1 \cup \Omega_p^2. $$
Hence, $r_{\min}\geq r^*,$ which completes the proof.
\end{proof}

Note that $Q_r(\vx)$ is convex, so we can introduce the subgradient method to make a relaxation.
Let $ g : \mathbb{R}^n \to \mathbb{R}$ be a convex function, and $ \vx\in \mathbb{R}^n $ be a point in the domain $ \text{dom} (g),$ a vector $ \vs \in \mathbb{R}^n $ is called a \textbf{subgradient} of $ g $ at $ \vx $ if it satisfies 
\begin{equation}\label{def: subdiff}
g(\vy) \geq g(\vx) + \langle \vs,\vy - \vx \rangle, \,\, \forall \,\vy \in \text{dom}(g),
\end{equation}
where $\langle \vv,\vb \rangle $ is the inner product of vector $\vv$ and $\vb.$ 
The subdifferential of $g$ at $\vx$ is a set consisting of all the vectors $\vy$ satisfying inequality \eqref{def: subdiff}, denoted as $\partial g$.
Since $Q_{r}(\vx)$ is convex and homogeneous of degree one, then for all $\vy\in\mathbb{R}^n,$ 
\begin{equation}\label{eq:Q relaxation}
\langle \vy, \vs \rangle \leq Q_{r}(\vy),\,\,\forall\, \vs \in \partial Q_{r}(\vx).
\end{equation}
Plugging the relaxation \eqref{eq:Q relaxation} into the subproblem \eqref{eq:sub two-step}, we get a relaxed version of \eqref{eq:sub two-step}. 
The domain $\Omega_p^1 \cup \Omega_p^2$ can be relaxed to $\{\vx:\|\vx\|_p=1\}.$
This modifies the two-step Dinkelbach iterative scheme into the following three-step one:
\begin{subequations}
\label{ite:three-step}
\begin{numcases}{}
\vx^{k+1}=\argmin\limits_{\| \vx\|_p =1} 
\{ \|\vx\|_\infty -\langle \vx,\vs^k \rangle\},~p\in[1,\boldsymbol{\infty}], \label{eq:subproblem}\\
r^{k+1}= r(\vx^{k+1}), \label{eq:r sequence}\\
\vs^{k+1}\in \partial Q_{r^{k+1}}(\vx^{k+1}). \label{eq；subg}
\end{numcases}
\end{subequations}
It has been pointed out that the subproblem \eqref{eq:subproblem}
 admits an analytical solution \cite{2024maxcutSI}, with $\vx^{k+1}$ remaining non-constant (see Algorithm \ref{alg: exact solution} and Remark \ref{remark: nonconstant}). 
Moreover, we are able to prove that such an iterative scheme still keeps the monotonicity (see Theorem \ref{them:monotonicity}) and has local convergence (see Theorem \ref{them:local convergence}).

\begin{algorithm}[htb]
\caption{ $\mathbf{DSI}$ algorithm for digraph conductance}
\label{alg:DSI}
\begin{algorithmic}
\STATE{\bf{Input:} an initial vector $\vx^1$, iterative number $T$}
\STATE{Set $k=1, \vx^*=\vx^1,r^*=r^1=r(\vx^1)$}
\WHILE{$k \leq T$}
\STATE{$\vs^k \gets \partial Q_{r^k}(\vx^k)$ (Algorithm \ref{alg:subgradient selection})}
\STATE{$\vx^{k+1}=\argmin\limits_{\|\vx\|_p=1} \{ \|\vx\|_\infty -\langle \vx,\vs^k \rangle\},~p\in[1,\boldsymbol{\infty}]$ (Algorithm \ref{alg: exact solution})}
\STATE{$r^{k+1} \gets (\vol(V)\|\vx^{k+1}\|_\infty -I^+(\vx^{k+1})-J(\vx^{k+1}))/2N(\vx^{k+1})$}
\IF{$r^{k+1}<r^*$}
\STATE{$r^*\gets r^{k+1}$, $\vx^*\gets \vx^{k+1}$, $k \gets k+1$}
\ELSE 
\STATE{break}
\ENDIF
\ENDWHILE
\RETURN $\vx^*,r^*$
\end{algorithmic}
\end{algorithm}

\begin{theorem}\label{them:monotonicity}
The sequence $\{r^{k}\}$ generated by Algorithm \ref{alg:DSI} satisfies $r^{k+1} \leq r^{k}$. 
\end{theorem}

\begin{proof}
Given the definitions of  $r^{k+1}$ and $Q_{r^{k}}(\vx^k)$, we deduce that
\begin{align*}
r^{k+1}-r^k &=\frac{\vol(V)\|\vx^{k+1}\|_{\infty}-I^+(\vx^{k+1})-J(\vx^{k+1})-2r^k N(\vx^{k+1}) }{2N(\vx^{k+1})} \\
&=\frac{\vol(V)\|\vx^{k+1}\|_{\infty}-\vol(V) Q_{r^k}(\vx^{k+1}) }{2N(\vx^{k+1})} \\
&=\frac{\vol(V)}{2N(\vx^{k+1})}(\|\vx^{k+1}\|_\infty-Q_{r^k}(\vx^{k+1})).
\end{align*}
By the definition of  $\vx^{k+1},$ it follows that
\begin{align*}
0=\|\vx^k\|_\infty- \langle \vx^k,\vs^k \rangle \geq \|\vx^{k+1}\|_\infty-\langle \vx^{k+1},\vs^{k} \rangle.
\end{align*}
Since $Q_{r^k}(\vx^k)$ is convex and homogeneous of degree one, we have 
\begin{equation*}
\|\vx^{k+1}\|_\infty \leq \langle \vx^{k+1}, \vs^k \rangle \leq Q_{r^k}(\vx^{k+1}).
\end{equation*}
Therefore, $r^{k+1}\leq r^k.$
\end{proof}

\subsection{Subgradient Selection}

Since $\partial Q_{r^k}(\vx^k)$ is an interval, and the selection of subgradient $\vs^k\in\partial Q_{r^k}(\vx^k)$ will affect the decreasing speed of $\{r^k\}$ in the iteration framework \eqref{ite:three-step}. This relation is illustrated in the following Lemma.

\begin{lemma}[Lemma 3.1 in \cite{2024maxcutSI}]
\label{lem:realation rsL and decrease}
Suppose $\vx^k$, $r^k$ and $\vs^k$ are generated by the iteration \eqref{ite:three-step}.
Then for $k\geq 1$, we always have $ \Vert \vs^k\Vert_1 \geq 1$. In particular, 

(1) $ \Vert \vs^k\Vert_1= 1$ if and only if $\vx^k/\|\vx^k\|_\infty 
\in \sgn(\vs^k),$ where \begin{equation}
\sgn(t)=\left\{\begin{array}{ll}
{ \{1\},} & {\text{ if } t> 0,} \\ 
{[-1,1],} & {\text{ if } t= 0,} \\ 
{\{-1\},} & {\text{ if } t<0}.\\
\end{array} \right.\\
\end{equation}

(2) $ \Vert \vs^k\Vert_1>1 $ if and only if $L(\vs^k)<0,$
where \begin{equation}\label{eq:L}
L(\vs) := \min_{\|\vx\|_p=1}\{ \|\vx\|_\infty -\langle \vx,\vs \rangle\},\,\, \vs\in\mathbb{R}^n. \end{equation}
In addition, if $\| \vs^k \|_1 > 1$, then $ r^{k+1}< r^k$. 
\end{lemma}

Now we illustrate how to choose an ideal subgradient to output a local optimum.
According to ~\eqref{eq:Q}, we have
\begin{align}\label{eq:subgradient expression}
\vs^k =\frac{1}{\vol(V)}(\vu^k+\vy^k+2r^k\vv^k) \in \partial Q_{r^k}(\vx^k),
\end{align}
where $\vu^k\in\partial I^+(\vx^k),$ $ \vy^k \in \partial J(\vx^k),$ $ \vv^k\in\partial N(\vx^k).$
Lemma \ref{lem:realation rsL and decrease} shows that the ideal $\vs^k$ satisfies that $\|\vs^k\|_1 >1,$ which will lead to $ r^{k+1}<r^k.$ Therefore, we select the subgradient $\vs^k \in \partial Q_{r^k}(\vx^k)$ to make $\|\vs^k\|_1$ as large as possible. $\partial Q_{r^k}(\vx^k)$ is an interval, which derives from the fact that $\partial I^+(\vx^k)$, $\partial J(\vx^k)$ and $\partial N(\vx^k)$ all consist of intervals. Their calculations are detailed as follows.
For convenience, the superscript $k$ is omitted for simplicity in the rest of this section if there is no ambiguity.	

$\bullet$ Characterization of $\partial I^+(\vx)$
	
\begin{equation}
(\partial I^+(\vx))_i = [p_i-q_i,p_i+q_i],
\end{equation}
where 
\begin{align}
p_i &= \sum_{j:i\to j \text{ or } j \to i} w_{ij}\sign(x_i+x_j)-q_i, \label{eq:p}\\
q_i &= \sum_{j\in \nen(i,\vx)} w_{ij}, \label{eq:q}\\
\sign(t)&=\left\{\begin{array}{ll}
{ 1,} & {\text{ if } t\geq 0,} \\ 
{-1,} & {\text{ if } t <0,}\\\end{array} \right.\\
\nen(i,\vx) &= \{j:i\to j \text{ or } j \to i \text{ and }x_i+x_j=0\}.
\end{align}


$\bullet$ Characterization of $\partial J(\vx)$

\begin{equation*}
\partial J(\vx)_i=\left\{\begin{array}{ll}
{l_i,} & {\text { if } J\neq 0,} \\ 
{[-|d^\delta_i|,|d^\delta_i|],} & {\text { if } J=0,}\\
\end{array}\right.\quad l_i=d_i^\delta \sign(\sum_i d_i^\delta x_i).
\end{equation*}

$\bullet$ Characterization of $\partial N(\vx)$

Let
\begin{align}
\text{median}(\vx)&=\argmin_{c\in\mathbb{R}}\sum_{i\in V}d_i|x_i-c|,\label{eq:median}\\
S^{\pm}(\vx)&=\{i\in V \big| x_i=\pm \|\vx\|_{\infty}\}, \label{eq:x_category1}\\
S^<(\vx)&=\{i\in V\big | |x_i|< \|\vx\|_{\infty}\},\label{eq:x_category2}\\
S^\alpha(\vx)&=\{i\in V \big |x_i=\alpha\},\quad\alpha\in \text{median}(\vx).\label{eq:x_category_alpha}
\end{align}

Then for a fixed $\alpha,$ the $i$-th component of $(\partial N(\vx))_i$ satisfies 
\begin{equation}
(\partial N(\vx))_i = [a_{i}^L, a_{i}^R], 
\end{equation}
where 
\begin{align}
a_{i}^L&=\left\{\begin{array}{ll}
{\max\{A-B+d_i,-d_i\},} & {\text { if } i\in S^{\alpha}(\vx)\text{ and }|	S^\alpha(\vx)|\geq 2,} \\ 
{A,} & {\text { if } i\in S^{\alpha}(\vx)\text{ and }|	S^\alpha(\vx)|\leq 1,}\\
{d_i\sign(x_i-\alpha),} & {\text { if }i\in V\setminus S^{\alpha}(\vx),}
\end{array}
\right. \label{eq:a_bound_l}\\
a_{i}^R &= \left\{\begin{array}{ll}
{\min\{A+B-d_i,d_i\},} & {\text { if } i\in S^{\alpha}(\vx)\text{ and }|	S^\alpha(\vx)|\geq 2,} \\ 
{A,} & {\text { if } i\in S^{\alpha}(\vx)\text{ and }|	S^\alpha(\vx)|\leq 1,}\\
{d_i\sign(x_i-\alpha),} & {\text { if }i\in V\setminus S^{\alpha}(\vx),}
\end{array}
\right. \label{eq:a_bound_r}	 \\
A &=\sum_{x_i<\alpha}d_i-\sum_{x_i>\alpha}d_i,\quad
B = \sum_{x_i=\alpha}d_i. \label{eq:v_ab}
\end{align}

In order to search for a subgradient that ensures $\|\vs^k\|_1>1 $ on the boundary of $\partial Q_r(\vx)$, we introduce a boundary indicator $\vb=(b_1,\ldots,b_n)\in\mathbb{R}^n,$
\begin{align}\label{eq:b}
b_{i}&=\left\{\begin{array}{ll}
{p_{i} +l_i+2r a_i+\chi_i q_i,} & {\text { if } J\neq 0,} \\ 
{p_{i}+ {\chi}_i |d_i^\delta|+ 2 r a_i+\chi_i q_{i},} & \text { if } J=0,
\end{array}\right.
\end{align}
with
\begin{align}\label{hh}	
a_i&=\begin{cases}
d_i\sign(x_i-\alpha),&\text{if $i\in V\setminus S^{\alpha}(\vx)$},\\
a^L_i,&\text{if $i\in S^{\alpha}(\vx)\cap S^+(\vx)$ and $|S^\alpha(\vx)|\geq 2$},\\
a^R_i,&\text{if $i\in S^{\alpha}(\vx)\cap S^-(\vx)$ and $|S^\alpha(\vx)|\geq 2$},\\ 
\argmax\limits_{t\in\{a^L_i,a^R_i\}}\{|p_i+l_i+2rt|\}, &\text{if $i\in S^{\alpha}(\vx)\cap S^<(\vx)$ and $|S^\alpha(\vx)|\geq 2$ and $J\neq 0 $},\\
\argmax\limits_{t\in\{a^L_i,a^R_i\}}\{|p_i+2rt|\}, &\text{if $i\in S^{\alpha}(\vx)\cap S^<(\vx)$ and $|S^\alpha(\vx)|\geq 2$ and $J= 0 $},\\
A,&\text{if $i\in S^{\alpha}(\vx)$ and $|S^\alpha(\vx)|\leq 1$},
\end{cases}
\end{align}
and 
\begin{equation}\label{eq:chi}
\chi(i)=\left\{\begin{array}{ll}
{\mp 1,} & {\text { if } x_i\in S^{\pm}(\vx),} \\ 
{\sign(p_i+l_i+2ra_i),} & {\text{ If } x_i\in S^{<}(\vx) \text{ and } J\neq0 ,}\\
{\sign (p_i+2ra_i),} & {\text{ If }x_i\in S^{<}(\vx) \text{ and } J=0 .}\\
\end{array} \right.\\
\end{equation}
Each $b_i$ sits on the boundary of $(\partial Q_r(\vx))_i$.
Let $\Sigma_{\vb}(\vx)$ denote a collection of permutations of $\{1,2,\ldots,n\}$ such that for any $\sigma\in\Sigma_{\vb}(\vx)$, it holds $|b_{\sigma(1)}|\leq |b_{\sigma(2)}|\leq \cdots\leq |b_{\sigma(n)}|$.
Consequently, for any $i\in V$, it can be easily verified that 
\begin{equation}
\label{eq:desire_v}
\exists\, 1 \leq i \leq n,\, \text{ s.t. } \frac{x_i}{\|\vx\|_{\infty}}\notin\sgn(s_i) \Rightarrow \frac{x_i}{\|\vx\|_{\infty}}\notin\sgn(b_i).
\end{equation}
The steps of subgradient selection are given in Algorithm \ref{alg:subgradient selection}.
Let 
\begin{equation}\label{eq:Vb}
V_{\vb}=\argmax\limits_{i}\{ b_i \chi_i:\,b_i \chi_i>0\}.
\end{equation}
The iteration will stop if $V_{\vb}=\emptyset$, and the iterative values \eqref{eq:r sequence} will strictly decrease when $V_{\vb}\neq\emptyset.$

\begin{theorem}[$\mathbf{DSI}$: strict descent guarantee]\label{them:subgradient select strict decrease}
$V_b\neq \emptyset$ if and only if  $ \exists\, \vs \in \partial Q_{r}(\vx) \text{~such that~} {L}(\vs) < 0.$ Moreover, if $V_b\neq \emptyset,$ then ${L}(\vs^k) < 0,$ leading to $r^{k+1}<r^k$ due to Lemma \ref{lem:realation rsL and decrease},
where $\vs^k$ is selected through Algorithm \ref{alg:subgradient selection} at the $k$-th iteration.
\end{theorem}
\begin{proof}
On the one hand, suppose $V_b\neq \emptyset$ and select $i \in V_b$ randomly.
Then we can select a subgradient $\vs\in \partial Q_r(\vx)$ such that $s_i = b_i/\vol(V).$ (The details of the selection steps are given in Algorithm \ref{alg:subgradient selection}).
Consequently,
\begin{equation}
\left\{
\begin{aligned} 
&x_{i}s_{i}<0, &\text{if }i\in S^{\pm}(\vx),\\
&	|s_{i}|>0, &\text{if }i\in S^<(\vx),\\
\end{aligned} \right.
\end{equation}
which directly yields ${L}(\vs)<0$ through Lemma \ref{lem:realation rsL and decrease}.


On the other hand, suppose that we can select a subgradient $\vs\in \partial Q_{r}(\vx)$ satisfying ${L}(\vs)<0$. According to Lemma~ \ref{lem:realation rsL and decrease}, there exists an $i^*\in\{1,2,\ldots,n\}$ satisfying
\begin{equation}
\left\{
\begin{aligned}
&\frac{b_{i^*}}{\vol(V)}\leq s_{i^*}<0, &\text{if }i^*\in S^+(\vx^k),\\
&\frac{b_{i^*}}{\vol(V)}\geq s_{i^*}>0, &\text{if }i^*\in S^-(\vx^k),\\
&	|b_{i^*}|\geq |s_{i^*}|>0, &\text{if }i^*\in S^<(\vx^k),\\
\end{aligned}
\right.
\end{equation}
and thus $\max\limits_i b_i\chi_i >0,$ which proves $V_b \neq \emptyset.$ 


\end{proof}

\begin{algorithm}[htb]
\caption{Subgradient selection}
\label{alg:subgradient selection}
\begin{algorithmic}
\STATE{\bf{Input:}  $G=(V,A),$ $\vp$  \eqref{eq:p}, $\vq $ \eqref{eq:q},  $\va$ \eqref{hh}, $\vec \chi$ \eqref{eq:chi}, $\sigma\in\Sigma_{\vb}(\vx)$, $V_b$ \eqref{eq:Vb}}
\STATE{ // select $\vu$}
\STATE{randomly select $i^*\in V_{\vb}$, $u_{i^*}=p_{i^*}+\sum\limits_{j\in \nen(i^*,\vx)}w_{i^*j}\chi(i^*)$}
\FOR {each neighbor $i\in \nen(i^*,\vx)$} 
\STATE{$u_i=p_i+w_{ii^*}\chi(i^*)+\sum\limits_{j\in \nen(i,\vx)\backslash\{i^*\}}w_{ij}\chi(\argmax\limits_{t\in\{i,j\}}\{\sigma^{-1}(t)\})$}
\ENDFOR
\FOR {each remaining vertex $i$}
\STATE{$u_i = p_i+\sum\limits_{j\in \nen(i,\vx)}w_{ij}\chi(\argmax_{t\in\{i,j\}}\{\sigma^{-1}(t)\})$}
\ENDFOR
\STATE{// select $\vv$}
\IF {$|S^\alpha(\vx)|\leq 1$}
\STATE{$v_i=\begin{cases}
a_i&\text{if }i\in V\backslash S^{\alpha}(\vx)\\
\sum\limits_{i:x_i<\alpha }d_i - \sum\limits_{i:x_i>\alpha} d_i&\text{if }i\in S^{\alpha}(\vx)
\end{cases}$}
\ELSIF {$|S^\alpha(\vx)|\geq 2$}
\STATE{$j^*=\left\{\begin{array}{ll}{i^*} & {\text { if } i^*\in S^{\alpha}(\vx)} \\ 
{\argmax\limits_{t\in S^{\alpha}(\vx)}\{\sigma^{-1}(t)\}} & {\text { otherwise}}
\end{array}\right.$
\STATE{}
$v_i=\begin{cases}
a_i&\text{if }i\in \{j^*\} \cup V\backslash S^{\alpha}(\vx)\\
\frac{A-a_{j^*}}{B-d_{j^*}}d_i&\text{if }i\in S^{\alpha}(\vx)\backslash \{j^*\}
\end{cases}$}
\ENDIF
\STATE{// select $\vy$}
\STATE{$\vy=\begin{cases}
  \chi_{i^*}\vd^\delta  &\text{ if }J=0\\
  \sign(J)\vd^\delta &\text { otherwise}
\end{cases}$}
\STATE{$\vs=(\vu + \vy +2r \vv)/\vol(V)$}
\RETURN $\vs$
\end{algorithmic}
\end{algorithm}


\subsection{Local Convergence}

The proposed boundary-detected subgradient selection steps for the $\mathbf{DSI}$ algorithm ensure that the sequence $\{r^k\}$ strictly decreases whenever possible as shown in Theorem~ \ref{them:subgradient select strict decrease}.
Furthermore, we are able to prove that the $\mathbf{DSI}$ algorithm converges to a local minimum $\vx^*$ in $C_r$ within finite iterations from a discrete perspective, where
\begin{align}
C_r &= \{\vx\in\mathbb{R}^n \big| r(\vec y)\geq r(\vx), \,\, \forall\,\vec y\in\{R_i\vx: i\in\{1,\ldots,n\}\}\}, \label{eq:CT}\\
(R_i\vx)_j& =
\begin{cases}
x_j, & j\ne i,\\
-x_j, & j=i.
\end{cases}\label{eq:Tialp}
\end{align}

\begin{theorem}[local convergence]
\label{them:local convergence} Suppose that the output of the $\mathbf{DSI}$ algorithm converges to $\vx^*$, then $\vx^*\in C_r$.
\end{theorem}

\begin{proof} 
According to the solution selection procedure in Algorithm \ref{alg: exact solution}, 
$\mathbf{\mathbf{DSI}}$ must produce nonconstant binary valued solutions (Remark \ref{remark: nonconstant}). Namely, $\vec \vx^*/\|\vec \vx^*\|_\infty \in \{-1, 1\}^n$. 
For simplicity, the superscript $*$ is neglected in the proof. Suppose $\vx\notin C_r$, which means there exists $i\in V$ satisfying $r(R_i\vx)<r(\vx).$
Given $\|R_i\vx\|_\infty = \|\vx\|_\infty,$ we deduce that
\begin{equation}\label{eq:ass}
	r(R_i\vx)<r(\vx) \quad \Leftrightarrow \quad Q_{r}(\vx)-Q_{r}(R_i\vx)<0,\quad r=r(\vx).
\end{equation}	
Hence,
\begin{align}
I^+(\vx)-I^+(R_i\vx)=\pm\sum\limits_{j: j\to i\in A \text{ or } i\to j\in A} 2w_{ij}x_j,\,\,i\in S^\pm(\vx), \label{eq:delta_iplus} \\
\begin{aligned}
N(\vx)-N(R_i\vx)=-|\Delta-d_i x_i|+|\Delta+d_i x_i|=&\begin{cases}
2d_ix_i,&\text{if }\Delta> d_i\|\vx\|_{\infty},\\
2\Delta x_i,&\text{if }\Delta \le d_i\|\vx\|_{\infty},\\ 
-2d_i x_i,&\text{if }\Delta< -d_i\|\vx\|_{\infty},
\end{cases}
\end{aligned}
\end{align}
where $\Delta=\sum\limits_{j=1}^nd_jx_j-d_ix_i$. Thus, it can be derived that 
\begin{equation}\label{eq:mmm}
\begin{cases}
1\in\text{median}(\vx)\text{ and }1\in\text{median}(R_i\vx),&\text{if }\Delta> d_i\|\vx\|_{\infty},\\
x_i\in\text{median}(\vx)\text{ and }-x_i\in\text{median}(R_i\vx),&\text{if }\Delta\leq d_i\|\vx\|_{\infty},\\
-1\in\text{median}(\vx)\text{ and }-1\in\text{median}(R_i\vx),&\text{if }\Delta<-d_i\|\vx\|_{\infty}.
\end{cases}
\end{equation}	
We also have
\begin{align*}
J(\vx)-J(R_i\vx)=-|\kappa- d_i^\delta x_i|+|\kappa+ d_i^\delta x_i|=&\begin{cases}
2d_i^\delta x_i,&\text{if }\kappa> d_i^\delta\|\vx\|_{\infty},\\
2\kappa x_i,&\text{if }\kappa \le d_i^\delta\|\vx\|_{\infty},\\ 
-2 d_i^\delta x_i,&\text{if }\kappa< -d_i^\delta\|\vx\|_{\infty},
\end{cases}
\end{align*}
where $\kappa=\sum\limits_{j=1}^nd_j^\delta x_j-d_i^\delta x_i$. 
Let $p_i\in(\partial I^+(\vx))_i$ with 
$$p_i=\sum\limits_{j: j\to i\in A \text{ or } i\to j\in A} 2w_{ij}z_{j},\,\,\,z_{j} =x_j/\|\vx\|_{\infty}\in \sgn(x_i+x_j).$$
Then we choose $a_i\in(\partial N(\vx))_i$ and $l_i \in (\partial J(\vx))_i$ with 
$$a_i =\begin{cases}
d_i,&\text{if }\Delta> d_i\|\vx\|_{\infty},\\
\Delta,&\text{if }\Delta \le d_i\|\vx\|_{\infty},\\ 
-d_i,&\text{if }\Delta< -d_i\|\vx\|_{\infty},
\end{cases} \text{ and }\, l_i =\begin{cases}
d_i^\delta, &\text{if }\kappa> d_i^\delta\|\vx\|_{\infty},\\
\kappa,&\text{if }\kappa \le d_i^\delta\|\vx\|_{\infty},\\ 
- d_i^\delta,&\text{if }\kappa< -d_i^\delta\|\vx\|_{\infty},
\end{cases}$$
respectively. Accordingly, combining ~\eqref{eq:subgradient expression} and \eqref{eq:ass} together 
leads to 
\begin{equation}
\label{xisileq0}
x_is_i=\frac{x_i p_i+x_i l_i+ 2rx_ia_i}{\vol(V)}<0,
\end{equation}
which contradicts Lemma~ \ref{lem:realation rsL and decrease}.
\end{proof}

\section{Numerical Experiments}\label{sec exp}

We implement our $\mathbf{DSI}$ algorithm using the Julia language. To achieve better solution quality, we select the initial point $ \vx_1 $ in Algorithm \ref{alg:DSI} as the second smallest eigenvector of the nonlinear Laplacian proposed in \cite{2016nonlinearlap}. Unlike their approach of simply rounding the scaled eigenvector, our method is rounding-free and yields superior solution quality.
We compare our algorithm against other digraph conductance algorithms in minimizing $\varphi_D$ \eqref{eq:varphi_D}
across various application domains.
Since many algorithms are designed to partition a digraph into $k$ parts for any arbitrary integer $k,$ such as the CLSZ algorithm proposed in \cite{2020hermite_cluster} and the iterative algorithm presented in \cite{2024itealg_digraph},
their experimental results turn out to be unsatisfactory when $k=2.$ 
The Sweepcut algorithm introduced in \cite{2016nonlinearlap} is based on nonlinear digraph Laplacian and is originally designed for partitioning a graph into two parts.
Therefore, we only present the experimental results of our $\mathbf{DSI}$ algorithm and the Sweepcut algorithm.

\subsection{Experimental Results On Synthetic Dataset}
This Section mainly presents the experimental results of the $\mathbf{DSI}$ algorithm on synthetic datasets compared with Sweepcut algorithm.
We choose the DSBM (Directed Stochastic Block Model \cite{2020hermite_cluster}) to generate datasets for experiments.
DSBM is an extension of the traditional SBM (Stochastic Block Model \cite{Holland1983SBM}), which is a random model used to generate graphs with real clustering structures. This model controls the characteristics of the generated digraph through the following parameters: $k,n,p,q$ and $\eta.$
Here $k$ is the number of clusters and we set $k=2,$ $n$ is the size of each block, and $p,q,\eta$ are parameters used to generate edges. 
To be more precise, the set $ \mathcal{G}( n, p, q, \eta) $ consists of graphs generated as follows: every $ G \in \mathcal{G} $ is a digraph defined on vertex set $ V = \{1, \dots, N\} $, where $ N = 2 n $. These vertices belong to clusters $ C_1, C_{2} $, where $ |C_j| = n $ for $ j \in [2] $. For any pair of vertices $u$ and $ v$, if they belong to the same cluster, they are connected by an arc with probability $ p $; otherwise, they are connected with probability $ q $. Moreover, if $ u \in C_1$ and $ v \in C_2$ are connected, the direction of this arc is determined by $\eta$: the direction is set to be $ u \to v $ with probability $\eta$, and $ v \to u $ with probability $1-\eta$. The direction of an arc inside a cluster is chosen uniformly at random.

We generate digraphs sampled from DSBM with $N=2000$ and different $p,$ $q,$ and $\eta.$ To be more precise, we choose $p=q \in [0.005,0.006,0.008,0.10],$ $\eta\in [0,0.05,0.1,0.15,0.2,0.25,0.3]$. 
The experimental results are shown in Figure \ref{fig:results on DSBM} with respect to $\eta.$
Sweep is conductance computed by the Sweepcut algorithm \cite{2016nonlinearlap}.
$\mathbf{DSI}$ is the conductance computed by the $\mathbf{DSI}$ algorithm.
It is clear that for different $p,q$ and $\eta$, our $\mathbf{DSI}$ algorithm provides better solution quality.

\begin{figure}[htb]
 \centering
 \begin{subfigure}{0.45\textwidth}
 \includegraphics[width=\linewidth]{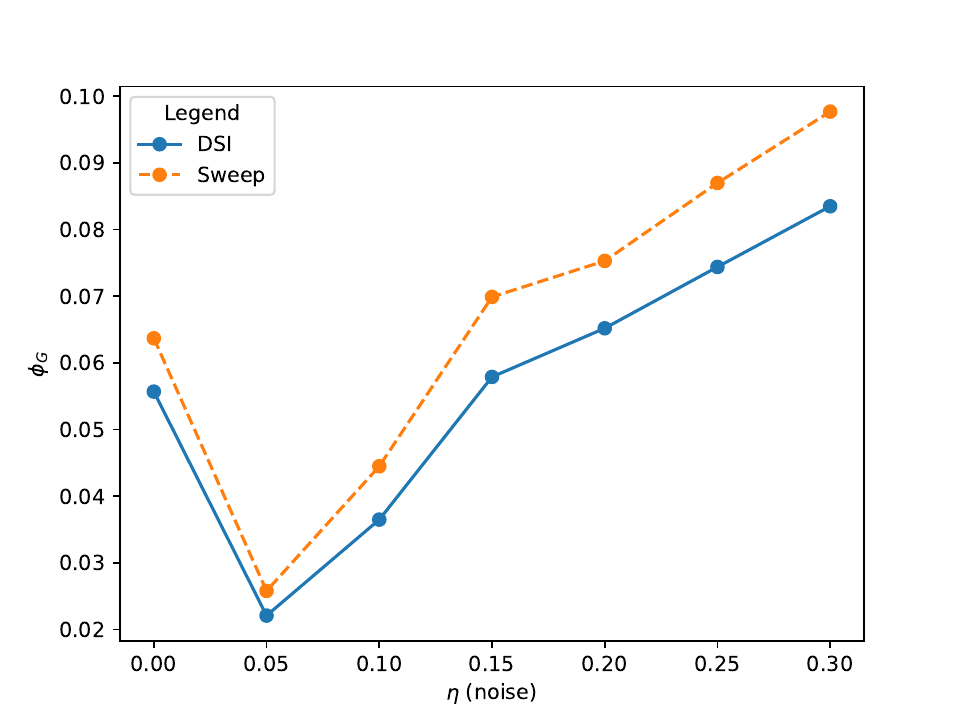}
 \caption{$p=q=0.005$}
 \end{subfigure}
 \hfill
 \begin{subfigure}{0.45\textwidth}
 \includegraphics[width=\linewidth]{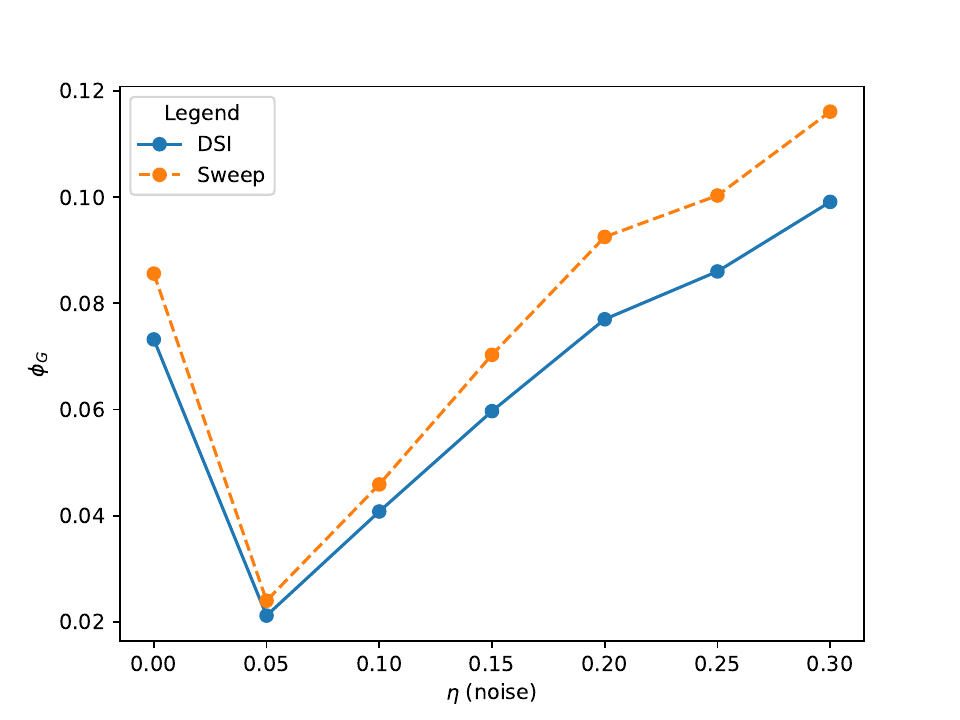}
 \caption{$p=q=0.006$}
 \end{subfigure}

 \begin{subfigure}{0.45\textwidth}
 \includegraphics[width=\linewidth]{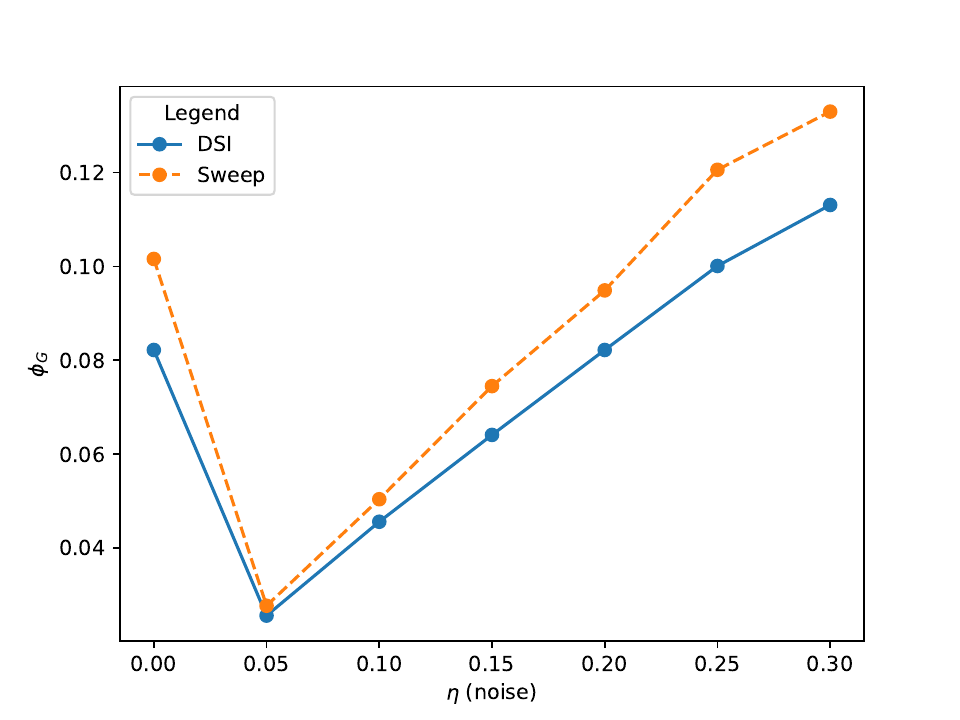}
 \caption{$p=q=0.008$}
 \end{subfigure}
 \hfill
 \begin{subfigure}{0.45\textwidth}
 \includegraphics[width=\linewidth]{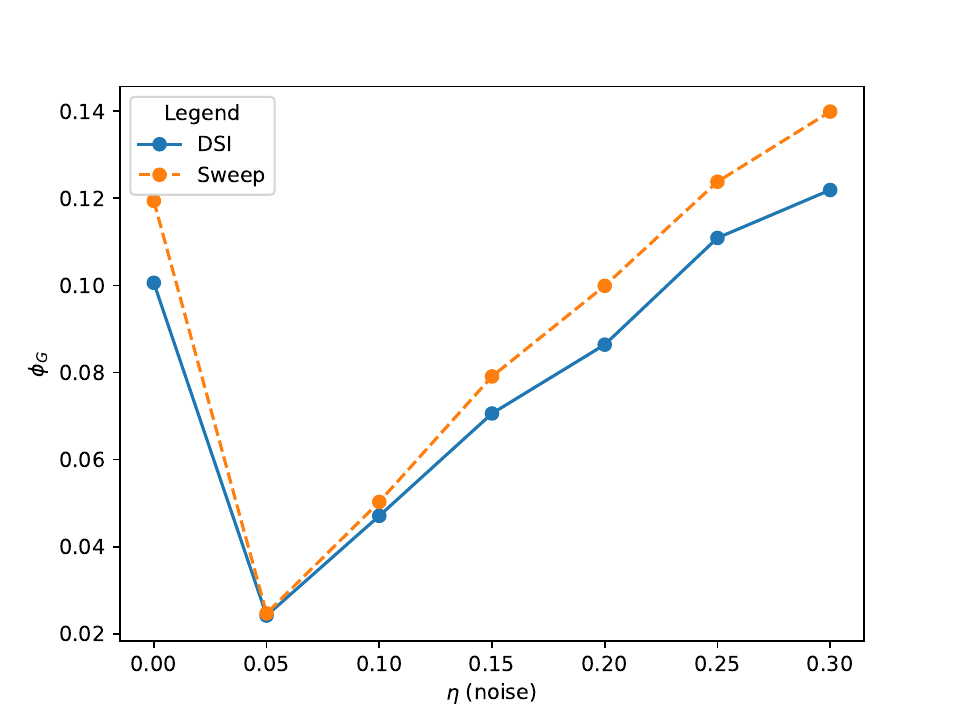}
 \caption{$p=q=0.01$}
 \end{subfigure}
 \caption{ Experimental results on DSBM: Figure (a) (b) (c) (d) illustrate the conductance values output by the $\mathbf{DSI}$ algorithm and the Sweepcut algorithm when $p=q=0.005,$ $p=q=0.006,$ $p=q=0.008$ and $p=q=0.01,$ respectively.}
 \label{fig:results on DSBM}
\end{figure}

\paragraph{Compare with Gurobi}

In order to better test the performance of the $\mathbf{DSI}$ algorithm, we will compare the experimental results of the $\mathbf{DSI}$ algorithm with the large-scale optimization solver Gurobi. 
Since Gurobi's binary vectors can only take values of 0 and 1, we model it as:
$$\varphi_D(G)=\min\limits_{\vx\in\{0,1\}^n\text{ and } \vx \neq \vec 0, \vx \neq \vec 1 } \frac{\sum\limits_{i \to j\in A} w_{ij}(x_i+x_j-2x_ix_j)}{\min\left\{\sum\limits_i^n d_i x_i,\sum\limits_i^n d_i (1-x_i)\right\}}.$$


\begin{table}[htb]
\centering
\caption{Numerical results of $\mathbf{DSI}$ and Gurobi on DSBM when $p=q=0.005$ and $n=2000.$}
\begin{tabular}{ccccc}
\toprule
$\eta$ & \textbf{$\mathbf{DSI}$} & \textbf{time} &\textbf{Gurobi} & \textbf{time} \\
\midrule
0.05 &0.0223 & 16.238 & 0.0682 & 120.0342 \\
0.10 &0.0379 & 18.215 & 0.0455 & 120.053 \\
0.15 &0.0575 & 16.816 & 0.0957 & 120.0339 \\
0.20 &0.0657 & 16.447 & 0.1667 & 120.0206 \\
0.25 &0.073 & 16.713 & 0.2047 & 120.0224 \\
0.30 &0.0824 & 16.798 & ${\inf}$ & 120.1067 \\
\bottomrule
\end{tabular}

\label{tab:my_label}
\end{table}
We set the time limit for Gurobi as 120s, inf means that Gurobi fails to output a solution.
Numerical results indicate that the $\mathbf{DSI}$ algorithm not only provides better solution quality than Gurobi, but also less time consuming.


\subsection{Experimental Results on Real Networks}
 
This section mainly presents the experimental results 
on real datasets. We first introduce the background of these real networks.

\paragraph{Florida Bay \cite{Benson2016HigherorderOO,Leskovec2014SNAPD}}

The Florida Bay dataset contains an unweighted digraph representing the carbon exchange between $ 125$ different species in Florida Bay. The arcs correspond to the typical diets of species observed within the bay.

\paragraph{C.elegans Network \cite{Kaiser2006NonoptimalCP}}
The C.elegans brain neural network describes the connections between neurons in the C.elegans brain. It is represented by a digraph containing 130 vertices and 610 arcs, where nodes represent neurons and arcs represent synaptic connections between neurons.

\paragraph{Telegram \cite{2020dataTelgram}} The telegram dataset is a pairwise influence network between $245$ Telegram channels with $8,912$ arcs.
 
\paragraph{Blog \cite{2005politicalblog}} The Blog dataset records $19,024$ arcs between $ 1,212$ political blogs from the 2004 US presidential election.  \\

Table \ref{tab:exp realnet} provides the parameters and experimental results of these real networks.
Here, $n$ is number of vertices, $m$ is the number of arcs, and $c_n, c_m$ are the number of vertices and arcs in the largest connected component.
$\mathbf{DSI}$ is the conductance computed by the $\mathbf{DSI}$ algorithm.
Sweep is conductance computed by Sweepcut algorithm.
It is clear that on different dataset, our $\mathbf{DSI}$ algorithm provides better solution quality.

\begin{table}[h!]
\centering
\caption{Numerical results of $\mathbf{DSI}$ and Sweepcut on real networks.}
\begin{tabular}{lcccccccc}
\toprule
name & $n$ & $m$ & $c_n$ & $c_m$ & $\mathbf{DSI}$ & {time 1} & $\text{Sweep}$ & {time 2} \\
\midrule
Celegans & 131 & 764 & 109 & 637 & 0.0126 & 5.953 & 0.0301 & 0.739 \\
Florida & 128 & 2106 & 103 & 1579 & 0.0087 & 6.4 & 0.0133 & 0.727 \\
blog & 1222 & 19024 & 793 & 15783 & 0.0286 & 8.678 & 0.0299 & 2.901 \\
telegram & 245 & 8912 & 138 & 4478 & 0.0209 & 8.633 & 0.0245 & 0.877 \\
\bottomrule
\end{tabular}
\label{tab:exp realnet}
\end{table}

\section{Conclusion and Outlook}

In this paper, we present an algorithm for estimating digraph conductance that addresses the fundamental asymmetry challenge inherent in directed networks. To overcome this challenge, we leverage properties of the Lov{\'a}sz extension to derive equivalent continuous formulations for both digraph conductance and conductance under submodular transformations, thereby bridging the discrete-continuous optimization gap. 
Building on this novel continuous formulation, we introduce the Directed Simple Iterative ($\mathbf{DSI}$) algorithm for digraphs with the Dinkelbach iterative framework. The outputs of $\mathbf{DSI}$ are rigorously proven to converge to binary local optima. 
Comprehensive evaluations on synthetic and real-world datasets demonstrate that $\mathbf{DSI}$ significantly outperforms benchmark methods including Sweepcut and Gurobi.
This work breaks through the limitations of symmetric Laplacian-based approaches for digraphs, opening new avenues for research in the analysis of directed complex networks. Notably, our framework demonstrates the potential for generalization to estimate hypergraph conductance and other related combinatorial metrics, highlighting its broad applicability.

\appendix \label{sec appendix}
\section{Exact solution of \eqref{eq:subproblem}}

The exact solutions for various choices of norm $p$ were presented in  \cite{2024maxcutSI}. Different choices of $p$ will not affect the numerical results too much. Therefore, in this paper, we only consider the case when norm $p=1$, and the specific steps are detailed in Algorithm \ref{alg: exact solution}.

\begin{algorithm}[htb]
\caption{Exact solution of subproblem \eqref{eq:subproblem}}
\label{alg: exact solution}
\begin{algorithmic}
\STATE{ \bf{Input:} $\vs^k$}
\STATE{Sort $|\vs^k|$ s.t. $|s_{\pi(1)}^k|\geq |s_{\pi(2)}^k| \geq\dots\geq |s_{\pi(n)}^k| \geq |s_{n+1}^k|=0$}
\STATE{ $A_m^k=\sum\limits_{j=1}^m(|s_{\pi(j)}^k|-|s_{\pi(m+1)}^k|)$, $m_0=\min\{m: A_m^k>1\}$, $m_1=\max\{m: A_{m-1}^k<1\}$ }
\IF{$\| \vec \vs^k \|_1 >1 $} 
\IF{$1\leq i \leq m_1$}
\STATE{ $z_{\pi(i)}=1$}
\ELSIF{$m_0 < i \leq n$}
\STATE{ $z_{\pi(i)}=0$}
\ELSIF{$m_1< i \leq m_0$}
\STATE{$z_{\pi(i)}\in [0,1]$}
\ENDIF
\STATE{ $x_i^{k+1}= \sign(s_i^k) z_i / \|\vec z\|_1$ for $i=1, 2, \dots, n$ }
\ELSE 
\STATE{$\vx^{k+1}\in \sgn(\vs^k)/n$}
\ENDIF
\RETURN $\vx^{k+1}$
\end{algorithmic}
\end{algorithm}

\begin{remark}\label{remark: nonconstant}
We illustrate that if $L(\vs^k)<0,$ then the next point output by Algorithm \ref{alg: exact solution}
 is  also a nonconstant vector.
 Let $L(\vx, \vs^k):= \|\vx\|_\infty - \langle \vx,\vs^k\rangle .$
 Note that every constant vector $\vx$  will lead to $L(\vx,\vs^k) \geq 0.$  
To be more precise, if $\vx = t\vec 1,$ for some $ t \in R,$ then  $I(\vx)=\|\vx\|_\infty  \vol(V),$ $J(\vx)=0,$  $N(\vx)=0$  and $Q_{r^k}(\vx)= \| \vx\|_\infty.$  Therefore,
\begin{equation}
L(\vx,\vs^k)=\|\vx\|_\infty - \langle \vx,\vs^k\rangle \geq \|\vx\|_\infty - Q_{r^k}(\vx)=0.
\end{equation}
Thus, if $L(\vs^k)<0,$  then $\vx^{k+1}$ of the iteration \eqref{ite:three-step} found by Algorithm \ref{alg: exact solution} is nonconstant and  this will lead to  $r^{k+1}<r^k.$
To be more precise,
\begin{equation}
0 >L(\vs^k) =L(\vx^{k+1},\vs^k) 
\geq \|\vx^{k+1}\|_\infty - Q_{r^k}(\vx^{k+1}) 
=\frac{2 r N(\vx^{k+1})}{\vol(V)} (r^{k+1} - r^k).
\end{equation}
Subsequently, we have $r^{k+1}< r^k.$
\end{remark}



\section{Proof of \eqref{eq:min formula}}
\label{proof of min formula}
\begin{proof}
On the one hand,  for any $S \subsetneq V,$ $\varphi_D(S) \leq \min\{\varphi^+(S),\varphi^-(S)\}.$
Therefore, $ \varphi_D(G) =\min\limits_{\emptyset \neq S \subsetneq V} \varphi_D(S) \leq \min\limits_{\emptyset \neq S \subsetneq V} \{ \varphi^+(S),\varphi^-(S)\} = \min\{\varphi^+(G),\varphi^-(G)\}.$ 

On the other hand, assume that $\varphi_D(S_0)=\varphi_D(G)$.
Without loss of generality,
assume that $\cut^+(S_0)\geq \cut^-(S_0).$
Subsequently, $\varphi_D(G)=\varphi_D(S_0) = \varphi^-(S_0) \geq \min\limits_{\emptyset \neq S \subsetneq V} \varphi^-(S) \geq \min\limits_{\emptyset \neq S \subsetneq V}\{\varphi^-(S),\varphi^+(S)\}.$ Therefore, $ \varphi_D(G)=\min\{ \varphi^+(G),\varphi^-(G)\}.$
\end{proof}

\bibliographystyle{siamplain}
\bibliography{references}
\end{document}